\documentclass[12pt]{amsart}

\usepackage{amstext}
\usepackage{amsthm}
\usepackage{amsmath}
\usepackage{latexsym}
\usepackage{amsfonts}
\usepackage{mathtools}
\usepackage[dvipsnames]{xcolor}

\newtheorem{theorem}{Theorem}
\newtheorem{lemma}[theorem]{Lemma}
\newtheorem{proposition}[theorem]{Proposition}

\theoremstyle{remark}\newtheorem{remark}[theorem]{Remark}
\theoremstyle{example}

\newcommand\CC{{\mathbb C}}
\newcommand\RR{{\mathbb R}}
\newcommand\DD{{\mathbb D}}

\newcommand\BB{{\mathbb B_n}}

\begin{document}
	\title[cyclic vectors in Dirichlet-type spaces]{A note on cyclic vectors in Dirichlet-type spaces in the unit ball of $\CC^n$}
	\author{Dimitrios Vavitsas}\email{dimitris.vavitsas@doctoral.uj.edu.pl}
	\address{Institute of Mathematics, Faculty of Mathematics and Computer Science, Jagiellonian University, \L ojasiewicza 6, 30-348 Krak\'ow, Poland}
	\thanks{Partially supported by NCN grant SONATA BIS no.  2017/26/E/ST1/00723 of the National Science Centre, Poland}
	
	\begin{abstract}
		We characterize model polynomials that are cyclic in Dirichlet-type spaces in the unit ball of $\mathbb C^n,$ and we give a sufficient capacity condition in order to identify non-cyclic vectors.
	\end{abstract}

	\maketitle
	\section{Introduction}
	
	Studying Dirichlet-type spaces in the unit ball of $\CC^n$ we can draw conclusions for classical Hilbert spaces of holomorphic functions such as the Hardy, Bergman and Dirichlet spaces. General introduction to this theory can be found in \cite{Rudin ball}, \cite{Zhu}.  
	
	The purpose of this note is to characterize model polynomials and to study special families of functions that are cyclic for the shift operators on these spaces. Moreover, we give a sufficient capacity condition in order to identify non-cyclic functions. Norm comparisons, sharp decay of norms for special subspaces, capacity conditions studied in \cite{Extrem polyn}, \cite{Extrem polyn II}, \cite{Ber}, \cite{Sola} are the main motivation for this work. The cyclicity of a function $f$ in a space of holomorphic functions is connected also with the problem of approximating $1/f$, see \cite{Sargent Sola}, \cite{Sargent Sola II} for the study of this subject.
	
	Full characterization of polynomials in more than two variables looks like a hard problem either in the unit ball or the polydisc. The cyclicity problem of polynomials for the bidisk was solved in \cite{Benetau} and shortly after extended in \cite{Knese}. The corresponding problem in the setting of the unit ball of $\CC^2$ was solved in \cite{Kosinski Vav}.
	
	\subsection{Dirichlet-type spaces in the unit ball}
	Denote the unit ball by
	$$\BB=\{z\in \CC^n:||z||<1\},$$
	and its boundary, the unit sphere by
	$$\mathbb S_n=\{z\in \CC^n:||z||=1\},$$
	where $||z||=\sqrt{|z_1|^2+...+|z_n|^2}$ is the associated norm of the usual \emph{Euclidean inner product} $\langle z,w \rangle=z_1\bar{w}_1+...+z_n\bar{w}_n.$ 
	Denote the \emph{class of holomorphic functions} in $\BB$ by $\textrm{Hol}(\BB).$ Any function $f\in \textrm{Hol}(\BB)$ has a power series expansion
	\begin{equation}\label{uniq power expansion in Reinh domain}
		f(z)=\sum_{k=0}^{\infty}a_kz^k=\sum_{k_1=0}^{\infty}... \sum_{k_n=0}^{\infty}a_{k_1,...,k_n}z_1^{k_1}\cdots z_n^{k_n}, \quad z\in \BB,
	\end{equation}
	where $k=(k_1,...,k_n)$ is a n-tuple index of non-negative integers,  $k!=k_1!\cdots k_n!$ and $z^k=z_1^{k_1}\cdots z_n^{k_n}.$
	The power series in ~\eqref{uniq power expansion in Reinh domain} exist, converges normal in $\BB$ and it is unique since the unit ball is a connected Reinhardt domain containing the origin, i.e. $(z_1,...,z_n)\in \BB$ implies $(e^{i\theta_1}z_1,...,e^{i\theta_n}z_n)\in \BB$ for arbitrary real $\theta_1,...,\theta_n,$ (see \cite{Hormander}).
	
	To simplify the notation we may write \eqref{uniq power expansion in Reinh domain} as follows:
	\begin{equation}
		f(z)=\sum_{m=0}^{\infty}\sum_{|k|=m}^{\infty}a_kz^k=\sum_{|k|=0}^{\infty}a_kz^k, \quad z\in \BB,
	\end{equation} 
	where $|k|=k_1+...+k_n.$
	
	Let $f\in \mathrm{Hol}(\BB)$. We say that $f$ belongs to the $\emph{Dirichlet-type}$ $space$ $D_{\alpha}(\BB),$ where $\alpha\in \mathbb{R}$ is a fixed parameter, if
	\begin{equation}\label{norm with sum}
		||f||^2_\alpha:=\sum_{|k|=0}^{\infty}(n+|k|)^\alpha\frac{(n-1)!k!}{(n-1+|k|)!}|a_{k}|^2<\infty.
	\end{equation}
	
	General introduction to the theory of Dirichlet-type spaces in the unit ball of $\CC^n$ can be found in \cite{Ahern}, \cite{Beatrous}, \cite{Li}, \cite{Michalska}, \cite{Sargent Sola}, \cite{Sola}, \cite{Zhu}. One variable Dirichlet-type spaces are discussed in the textbook \cite{Fallah-Tomas}. The weights in the norm in ~\eqref{norm with sum} are chosen in such a way that $D_0(\BB)$ and $D_{-1}(\BB)$ coincide with the $\emph{Hardy}$ and $\emph{Bergman}$ $\emph{spaces}$ of the ball, respectively. The $\emph{Dirichlet space}$ having M\"{o}bius invariant norm corresponds to the parameter choice $\alpha=n.$
	
	By the definition, $D_\alpha(\BB)\subset D_\beta(\BB),$ when $\alpha\geq \beta.$  Polynomials are dense in the spaces $D_{\alpha}(\BB),$ $\alpha \in \RR,$ and $z_i\cdot f\in D_\alpha(\BB),$ $i=1,...,n$ whenever $f\in D_\alpha(\BB).$ 
	
	A \emph{multiplier} in $D_\alpha(\BB)$ is a holomorphic function $\phi:\BB\rightarrow\CC$ that satisfies $\phi\cdot f\in D_\alpha(\BB)$ for all $f\in D_\alpha(\BB).$ Polynomials, as well as holomorphic functions in a neighbourhood of the closed unit ball, are multipliers in every space $D_\alpha(\BB)$. 
	
	\subsection{Shift operators and cyclic vectors}
	Consider the bounded linear operators $S_1,...,S_n:D_{\alpha}(\BB)\rightarrow D_{\alpha}(\BB)$ defined by $S_i:f\mapsto z_i\cdot f.$ We say that $f\in D_\alpha(\BB)$ is a \emph{cyclic vector} if the closed invariant subspace, i.e.
	$$[f]:=\mathrm{clos}\, \mathrm{span}\{z_1^{k_1}\cdots z_n^{k_n}f:k_1,...,k_n=0,1,...\}$$
	coincides with $D_\alpha(\BB)$ (the closure is taken with respect to the $D_\alpha(\BB)$ norm). An equivalent definition is that $f$ is cyclic if and only if $1\in [f].$
	
	Since $D_\alpha(\BB)$ enjoys the \emph{bounded point evaluation property} a function that is cyclic cannot vanish inside the unit ball. Thus, we focus on functions non-vanishing in the domain. Also, non-zero constant functions are cyclic in every space $D_\alpha(\BB).$ More information regarding cyclic vectors in Dirichlet-type spaces over the disk, the polydisc and the unit ball can be found in \cite{Extrem polyn}, \cite{Extrem polyn II}, \cite{Benetau}, \cite{Ber}, \cite{Brown Shields},  \cite{Fallah-Tomas}, \cite{Knese}, \cite{Kosinski Vav}, \cite{Sargent Sola II}, \cite{Sola}.
	
	Just as in the settings of the bidisk and the unit ball of two variables, the cyclicity of a function $f\in D_\alpha(\BB)$ is inextricably linked with its $\emph{zero set}$
	$$\mathcal{Z}(f)=\{z\in \CC^n:f(z)=0\}.$$
	The zeros of a function lying on the sphere are called the \emph{boundary zeros}.
	
	\subsection{Plan of the paper}
	Section~\ref{section Relation among} studies Dirichlet-type spaces. In particular, we give a crucial relation among them. Using fractional radial derivatives and the Cauchy formula of functions lying in the $\emph{ball algebra}$ $A(\BB)$ which contains functions that are continuous on the closed unit ball and holomorphic in its interior, we give an equivalent norm of Dirichlet-type spaces for a wide range of parameters $\alpha.$
	
	Section~\ref{section diagonal} studies diagonal subspaces. In particular, we extend result from \cite{Sola}. It makes sense to define functions $f\in \mathrm{Hol}(\BB)$ using functions $\tilde{f}\in \mathrm{Hol}(\DD(\mu))$ for a proper $\mu>0.$ Geometrically speaking, we are looking at a disk embedded in the ball but not in a coordinate plane. Thus, we may switch the problem of cyclicity from the ball to spaces of holomorphic functions of one variable that are well known. Then we use optimal approximants in order to identify cyclicity.
	
	Moreover, we prove cyclicity for model polynomials for proper parameters. In the setting of the unit ball of two variables, see \cite{Sola}, the model polynomials are the following: $1-z_1$ which vanishes in the closed unit ball on a singleton, i.e. $\mathcal{Z}(1-z_1)\cap \mathbb S_2=\{(1,0)\},$ and $1-2z_1z_2$ which vanishes along an analytic curve, i.e. $\mathcal{Z}(1-2z_1z_2)\cap \mathbb S_2=\{(e^{i\theta}/\sqrt{2},e^{-i\theta}/\sqrt{2}):\theta\in \RR\}.$ In our case, the corresponding candidates are the following:
	$$p(z)=1-m^{m/2}z_1\cdots z_m, \quad 1\leq m\leq n.$$
	They vanish in the closed unit ball along the following analytic sets:
	$$\mathcal{Z}(p)\cap \mathbb{S}_n=\{1/\sqrt{m}(e^{i\theta_1},..,e^{i\theta_{m-1}}, e^{-i(\theta_1+...+\theta_{m-1})},0,..,0):\theta_i\in \mathbb{R}\}.$$
	These polynomials are also studied with respect to the Drury-Arveson space in \cite{Sargent Sola}.
	
	In two variables, $1-z_1$ is cyclic in $D_\alpha(\mathbb B_2)$ precisely when $\alpha\leq 2,$ and $1-2z_1z_2$ is cyclic in $D_\alpha(\mathbb B_2)$ precisely when $\alpha\leq 3/2.$ Here, there are more than two fixed parameters. The characterization of cyclicity of these two polynomials was crucial in \cite{Kosinski Vav}.
	
	Section~\ref{section radial} studies the radial dilation of a polynomial. Using the equivalent norm of Section~\ref{section Relation among}, we identify cyclicity for the model polynomials via the powerful radial dilation method. In particular, we show that if $p/p_r\rightarrow 1$ weakly, where $p_r(z)=p(rz)$ is a radial dilation of $p,$ then $p$ is cyclic, (see \cite{Knese} for the bidisk settings and \cite{Kosinski Vav} for the unit ball in two variables). This method is quite interesting since it can be applied to an arbitrary polynomial. Note that in \cite{Knese}, \cite{Kosinski Vav}  the radial dilation method is one of the main tools of solving cyclicity problem for polynomials. The main result of this section verifies the arguments made about polynomials in Section~\ref{section diagonal}.
	
	Section~\ref{section capacity} studies non-cyclic vectors. We use the notion of Riesz $\alpha$-capacity in order to identify non-cyclic functions. Moreover, we study Cauchy transforms of Borel measures supported on zero sets of the radial limits of a given function $f\in D_\alpha(\BB)$ and we give asymptotic expansions of their norms. Then employing a standard scheme due to Brown and Shields, see \cite{Brown Shields}, we prove the main result. Note that this sufficient capacity condition for non-cyclicity in Dirichlet-type spaces in the unit ball of two variables was proved by A. Sola in \cite{Sola}.
	
	\section*{Standard tools}
	
	Let us give some standard tools which will be useful in the sequel.
	
	The binomial series: 
	$$\frac{1}{(1-x)^{\alpha}}=\sum_{k=0}^{\infty}\frac{\Gamma(k+\alpha)}{\Gamma(\alpha)k!}x^k,$$
	where $|x|<1$ is a complex number and $\alpha$ is a non-negative real number. The asymptotic behaviour of the $\Gamma$-function is the following: $\Gamma(k+\alpha)\asymp (k-1)!k^\alpha,$ where the symbol $\asymp$ denotes that the ratio of the two quantities either tends to a constant as $k$ tends to infinity or it is rather two sides bound by constants.
	
	The multinomial formula:
	$$(x_1+...+x_n)^k=\sum_{|j|=k}\frac{k!}{j!}x_1^{j_1}\cdots x_n^{j_n},$$
	where $j=(j_1,...,j_n)$ is a $n$-tuple index of non-negative integers and $x_i$ are complex numbers.
	
	The Stirling formula that describes the asymptotic behaviour of the gamma function:
	$$k!\asymp k^{1/2}k^k/e^k.$$
	
	Denote the normalized area measure on $\CC^n=\RR^{2n}$ by $du(z)$ and the normalized rotation-invariant positive Borel measure on $\mathbb S_n$ by $d\sigma(\zeta),$ (see \cite{Rudin ball}, \cite{Zhu}). The measures $du(z)$ and $d\sigma(\zeta)$ are related by the formula
	$$\int_{\CC^n}f(z)du(z)=2n\int_{0}^{\infty}\int_{\mathbb S_n}\epsilon^{2n-1}f(\epsilon\zeta)d\sigma(\zeta)d\epsilon.$$
	
	The holomorphic monomials are orthogonal to each other in $L^2(\sigma),$ that is, if $k$ and $l$ are multi-indices such that $k\neq l,$ then 
	$$ \int_{\mathbb S_n}\zeta^k \bar{\zeta}^l d\sigma(\zeta)=0.$$
	Moreover,
	$$\int_{\mathbb S_n}|\zeta^k|^2d\sigma(\zeta)=\frac{(n-1)!k!}{(n-1+|k|)!} \quad \text{and} \quad  \int_{\BB}|z^k|^2du(z)=\frac{n!k!}{(n+|k|)!}.$$
	
	\section{Relation among Dirichlet-type spaces and equivalent norms}\label{section Relation among}
	
	We study the structure of Dirichlet-type spaces. Note that
	$$R(f)(z)=z_1\partial_{z_1}f(z)+...+z_n\partial_{z_n}f(z)$$
	is the \emph{radial derivative} of a function $f.$ The radial derivative plays a key role in the function theory of the unit ball. A crucial relation among these spaces is the following.
	\begin{proposition}\label{Prop relation among D.S.}
		Let $f\in \mathrm{Hol}(\BB)$ and $\alpha\in \RR$ be fixed. Then
		\begin{equation*}
			f\in D_\alpha(\BB)\quad  \text{if and only if} \quad  n^q f+R^q (f)+q\sum_{i=1}^{q-1}n^iR^{q-i}(f)\in D_{\upsilon}(\BB),
		\end{equation*}
		where $\alpha=2q+\upsilon,$ $q\in \mathbb N$ and $R^q$ is the $q$-image of the operator $R.$
		\begin{proof}
			Indeed, it is enough to check that
			$$||nf+R(f)||^2_{\alpha-2}=\sum_{|k|=0}^{\infty}(n+|k|)^{\alpha-2}\frac{(n-1)!k!}{(n-1+|k|)!}(n+|k|)^2|a_{k}|^2=||f||_{\alpha}^2.$$
		\end{proof}
	\end{proposition}
	
	We continue by giving an equivalent characterization of Dirichlet-type norms. In Dirichlet-type spaces in the unit ball, one of the integral representations of the norm is achieved in a limited range of parameters.
	\begin{lemma}[see\cite{Michalska}]\label{le: equivalent int norm}
		If $\alpha\in (-1,1)$, then $||f||^2_\alpha$ is equivalent to $$|f|^2_\alpha:=\int_{\BB} \frac{||\nabla(f)(z)||^2 - |R(f)(z)|^2}{ (1-||z||^2)^\alpha} du(z).$$
	\end{lemma}
	Above, $\nabla(f)(z)=(\partial_{z_1}f(z),...,\partial_{z_n}f(z))$ denotes the \emph{holomorphic gradient} of a holomorphic function $f.$  Note that Proposition~\ref{Prop relation among D.S.} allows us to use Lemma~\ref{le: equivalent int norm} whenever $\upsilon\in (-1,1).$
	.
	Let $\gamma,t\in \RR$ be such that neither $n+\gamma$ nor $n+\gamma+t$ is a negative integer. If $f=\sum_{|k|=0}^{\infty}a_kz^k$ is the homogeneous expansion of a function $f\in \textrm{Hol}(\BB),$ then we may define an invertible continuous linear operator with respect to the topology of uniform convergence on compact subsets of $\BB,$ denoted by $R^{\gamma,t}: \textrm{Hol}(\BB)\rightarrow \textrm{Hol}(\BB)$ and having expression
	\begin{equation*}
		R^{\gamma,t}f(z)=\sum_{|k|=0}^{\infty}C(\gamma,t,k)a_kz^k, \quad z\in \BB,
	\end{equation*}
	where
	\begin{equation}\label{estimate of C(g,t,k)}
		C(\gamma,t,k)=\frac{\Gamma(n+1+\gamma)\Gamma(n+1+|k|+\gamma+t)}{\Gamma(n+1+\gamma+t)\Gamma(n+1+|k|+\gamma)}\asymp |k|^t.
	\end{equation}
	See \cite{Zhu} for more information regarding these fractional radial derivatives.
	
	\begin{lemma}\label{integral repres of R}
		Let $t\in \RR$ be such that $n-1+t \geq 0.$ If $f\in A(\BB),$ then
		$$R^{-1,t}f(z)=\int_{\mathbb S_n}\frac{f(\zeta)}{(1-\langle z,\zeta \rangle)^{n+t}}d\sigma(\zeta), \quad z\in \BB.$$
		\begin{proof}
			The continuous linear operator $R^{\gamma,t},$ see \cite{Zhu}, satisfies
			$$R^{\gamma,t}\Big(\frac{1}{(1-\langle z,w \rangle)^{n+1+\gamma}}\Big)=\frac{1}{(1-\langle z,w \rangle)^{n+1+\gamma+t}}$$
			for all $w\in \BB.$
			Next, define $f_\epsilon$ for $\epsilon\in (0,1)$ by
			$$f_\epsilon(z)=\int_{\mathbb S_n}\frac{f(\zeta)}{(1-\langle z,\epsilon\zeta \rangle)^n}d\sigma(\zeta), \quad z\in \BB.$$
			The Cauchy formula holds for $f\in A(\BB)$ and hence $f=\lim_{\epsilon\rightarrow 1^{-}}f_\epsilon.$
			It follows that
			\begin{align*}
				R^{-1,t}f(z)&=R^{-1,t}\Big(\lim_{\epsilon\rightarrow 1^{-}}\int_{\mathbb S_n}\frac{f(\zeta)}{(1-\langle z,\epsilon\zeta \rangle)^n}d\sigma(\zeta) \Big)\\
				&=\lim_{\epsilon\rightarrow 1^{-}}R^{-1,t}\Big(\int_{\mathbb S_n}\frac{f(\zeta)}{(1-\langle z,\epsilon\zeta \rangle)^n}d\sigma(\zeta) \Big)\\
				&=\lim_{\epsilon\rightarrow 1^{-}}\int_{\mathbb S_n}f(\zeta)R^{-1,t}\Big(\frac{1}{(1-\langle z,\epsilon\zeta \rangle)^{n}}\Big)d\sigma(\zeta)\\
				&=\lim_{\epsilon\rightarrow 1^{-}}\int_{\mathbb S_n}\frac{f(\zeta)}{(1-\langle z,\epsilon\zeta \rangle)^{n+t}}d\sigma(\zeta)\\
				&=\int_{\mathbb S_n}\frac{f(\zeta)}{(1-\langle z,\zeta \rangle)^{n+t}}d\sigma(\zeta)
			\end{align*}
			and the assertion follows.
		\end{proof}
	\end{lemma}
	
	\begin{theorem}\label{equivalent asym norm with cauchy formula}
		Let $\alpha\in \RR$ be such that $n-1+\alpha/2\geq 0$ and $f\in A(\BB).$ Then $f\in D_\alpha(\BB)$ if and only if
		$$\int_{\BB}(1-||z||^2)\Big|\int_{\mathbb S_n}\frac{f(\zeta)\bar{\zeta}_p}{(1-\langle z,\zeta \rangle)^{n+\alpha/2+1}}d\sigma(\zeta)\Big|^2du(z)<\infty$$
		and
		$$\int_{\BB}\Big|\int_{\mathbb S_n}\frac{(\overline{z_p\zeta_q-z_q\zeta_p})f(\zeta)}{(1-\langle z,\zeta \rangle)^{n+\alpha/2+1}}d\sigma(\zeta)\Big|^2du(z)<\infty,$$
		where $p,q=1,...,n.$
		\begin{proof}
			Choose $t$ so that $\alpha=2t.$	Note that $n,t$ are fixed and hence
			\begin{equation*}
				||f||^2_\alpha\asymp\sum_{|k|=0}^{\infty}\frac{(n-1)!k!}{(n-1+|k|)!}||k|^ta_{k}|^2.
			\end{equation*}
			Thus, \eqref{estimate of C(g,t,k)} implies that $||R^{-1,t}f||_{0}\asymp ||f||_\alpha.$ One can apply then the integral representation of Dirichlet-type norms to $R^{-1,t}f\in \textrm{Hol}(\BB),$ i.e. $||R^{-1,t}f||_{0}$ is equivalent to $|R^{-1,t}f|_{0}.$ 
			According to Lemma~\ref{integral repres of R} we get that
			$$\partial_{z_p}(R^{-1,t}f)(z)=\int_{\mathbb S_n}\frac{f(\zeta)\bar{\zeta}_p}{(1-\langle z,\zeta \rangle)^{n+t+1}}d\sigma(\zeta), \quad z\in \BB,$$
			where $p=1,...,n.$ Expand the term $||\nabla(f)||^2-|R(f)|^2$ as follows: 
			\begin{equation*}
				||\nabla(f)||^2-|R(f)|^2=(1-||z||^2)||\nabla(f)||^2+\sum_{p,q}|\bar{z}_p\partial_{z_q}f-\bar{z}_q\partial_{z_p}f|^2.
			\end{equation*}
			The assertion follows by Lemma~\ref{le: equivalent int norm}.
		\end{proof}
	\end{theorem}
	
	\section{Diagonal subspaces}\label{section diagonal}
	
	In \cite{Extrem polyn}, a method of construction of optimal approximants via determinants in Dirichlet-type spaces in the unit disk is provided. Similarly, we may define optimal approximants in several variables, (see \cite{Sargent Sola}).
	
	Fix $N\in \mathbb N.$ We define the space of polynomials $p\in \CC[z_1,...,z_n]$ with degree at most $nN$ as follows:
	$$P_N^n:=\{p(z)=\sum_{k_1=0}^{N}... \sum_{k_n=0}^{N}a_{k_1,...,k_n}z_1^{k_1}\cdots z_n^{k_n}\}.$$
	
	\begin{remark}
		Let $(X,||\cdot||)$ be a normed space and fix $x\in X,$ $C\subset X.$ The distance between $x$ and the set $C$ is the following:
		$$\mathrm{dist}_X(x,C):=\inf\{||x-c||:c\in C\}.$$
		It is well known that if $X$ is a Hilbert space and $C\subset X$ a convex closed subset, then for any $x\in X,$ there exists a unique $y\in C$ such that $||x-y||=\mathrm{dist}_X(x,C).$ Let $f\in D_\alpha(\BB)$ be non-zero constant. We deduce that for any $N\in \mathbb N,$ there exists exactly one $p_N\in P_N^n$ satisfying
		$$||p_Nf-1||_\alpha=\mathrm{dist}_{D_\alpha(\BB)}(1,f\cdot P_N^n).$$ 
	\end{remark}
	
	Let $f\in D_\alpha(\BB).$ We say that a polynomial $p_N\in P^n_N$ is an $\emph{optimal}$ $\emph{approximant}$ $\emph{of order}$ $N$ to $1/f$ if $p_N$ minimizes $||pf-1||_\alpha$ among all polynomials $p\in P_N^n.$ We call $||p_Nf-1||_\alpha$ the $\emph{optimal norm of order}$ $N$ associated with $f.$
	
	Let $M=(M_1,...,M_n)$ be a multi-index, where $M_i$ are non-negative integers, and $m\in \{1,...,n\}.$ Setting
	$$\mu(m):=\frac{(M_1+...+M_m)^{M_1+...+M_m}}{M_1^{M_1}\cdots M_m^{M_m}},$$
	we see that
	\begin{equation}\label{mu kai z}
		\mu(m)^{1/2}|z_1|^{M_1}\cdots |z_m|^{M_m}\leq 1, \quad z\in \BB.
	\end{equation}
	
	Using \eqref{mu kai z} we may construct polynomials that vanish in the closed unit ball along analytic subsets of the unit sphere.
	
	\begin{remark}
		Let $\tilde{f}\in \mathrm{Hol}(\DD(\mu(m)^{-1/4})),$ where 
		$$\DD(\mu)=\{z\in \CC:|z|<\mu\}, \quad \mu> 0.$$
		According to \eqref{mu kai z}  we define the following function:
		$$f(z)=f(z_1,...,z_n)=\tilde{f}(\mu(m)^{1/4}z_1^{M_1}\cdots z_m^{M_m}), \quad z\in \BB.$$
		Then $f\in \mathrm{Hol}(\BB)$ and it depends on $m$ variables. Note that we may change the variables $z_1,...,z_m$ by any other $m$ variables. For convenience, we choose the $m$ first variables. The power $1/4$ will be convenient in the sequel.
	\end{remark}
	
	Thus, the question that arises out is if we may define closed subspaces of $D_\alpha(\BB)$ passing through one variable functions.	We shall see that these subspaces are called diagonal subspaces due to the nature of the power series expansion of their elements. 
	
	Instead of the classical one variable Dirichlet-type spaces of the unit disk, we may consider spaces $d_\beta,$ $\beta\in \RR,$ consisting of holomorphic functions $\tilde{f}\in \mathrm{Hol}(\DD(\mu^{-1/4})).$ Moreover, such functions with power series expansion $\tilde{f}(z)=\sum_{l=0}^{\infty}a_lz^l$
	are said to belong to $d_\beta$ if
	$$||\tilde{f}||^2_{d_\beta}:= \sum_{l=0}^{\infty}\mu^{-l/2}(l+1)^\beta|a_l|^2<\infty.$$
	
	There is a natural identification between the function theories of $D_\beta(\DD)$: one variable Dirichlet-type spaces of the unit disk, and $d_\beta,$ and one verifies that the results in \cite{Extrem polyn} are valid for $d_\beta.$
	
	We are ready to define diagonal closed subspaces. Set $$\beta(\alpha):=\alpha-n+\frac{m+1}{2}.$$
	
	Let $\alpha,$ $M,$ $m$ be as above. The diagonal closed subspace of $D_\alpha(\BB)$  is the following:
	$$J_{\alpha,M,m}:=\{f\in D_\alpha(\BB):\exists \tilde{f}\in d_{\beta(\alpha)}, f(z)= \tilde{f}(\mu(m)^{1/4}z_1^{M_1}\cdots z_m^{M_m})\}.$$
	
	The existence of a holomorphic function $\tilde{f}$ is unique by identity principle and hence there is no any amiss in the definition. Any function $f\in J_{\alpha,M,m}$ has an expansion of the form
	$$f(z)=\sum_{l=0}^{\infty}a_l(z_1^{M_1}\cdots z_m^{M_m})^l.$$
	
	The relation of norms between one variable and diagonal subspaces follows.
	\begin{proposition}\label{relation norm one variable diagonal}
		If $f\in J_{\alpha,M,m},$ then $||f||_\alpha\asymp ||\tilde{f}||_{d_{\beta(\alpha)}}.$
		\begin{proof}
			If $f\in J_{\alpha,M,m},$ then
			$$||f||^2_\alpha \asymp \sum_{l=0}^{\infty}(l+1)^\alpha\frac{(M_1l)!\cdots (M_ml)!}{(n-1+(M_1+\cdots M_m)l)!}|a_l|^2.$$
			By Stirling's formula, we obtain
			$$||f||^2_\alpha \asymp \sum_{l=0}^{\infty}(l+1)^{\alpha-n+m/2+1/2}\mu(m)^{-l}|a_l|^2.$$
			On the other hand, define the function
			$f'(z)=\sum_{l=0}^{\infty}\mu(m)^{-l/4}a_lz^l.$
			Then $f'(\mu(m)^{1/4}z_1^{M_1}\cdots z_m^{M_m})=f(z_1,...,z_n)$ and
			$$||f'||^2_{d_{\beta(\alpha)}}\asymp \sum_{l=0}(l+1)^{\alpha-n+m/2+1/2}\mu(m)^{-l}|a_l|^2.$$
			The assertion follows since $f'$ coincides with $\tilde{f}.$
		\end{proof}
	\end{proposition}
	
	The corresponding Lemma 3.4 of \cite{Extrem polyn II} in our case is the following.
	\begin{lemma}\label{ineq of norms projection}
		Let $f\in J_{\alpha,M,m},$ where $\alpha,M,m$ be as above. Let $r_N\in P_N^n$ with expansion
		$$r_N(z)=\sum_{k_1=0}^{N}\cdots \sum_{k_n=0}^{N}a_{k_1,...,k_n}z_1^{k_1}... z_n^{k_n},$$
		and consider its projection onto $J_{\alpha,M,m}$
		$$\pi(r_N)(z)=\sum_{\{l:M_1l,...,M_ml\leq N\}}c_{M_1l,...,M_ml,0,...,0}z_1^{M_1l}\cdots z_m^{M_ml}.$$
		Then
		$$||r_Nf-1||_\alpha\geq ||\pi(r_N)f-1||_\alpha.$$
	\end{lemma}	
	
	Moreover, just as in Proposition~\ref{relation norm one variable diagonal}, there is a relation of optimal approximants between one variable and diagonal subspaces.
	\begin{proposition}
		If $f\in J_{\alpha,M,m},$ then
		$$\mathrm{dist}_{D_\alpha(\BB)}(1,f\cdot P_N^n)\asymp \mathrm{dist}_{d_\beta(\alpha)}(1,\tilde{f}\cdot P^1_N).$$
		\begin{proof}
			Let $r_N,$ $\pi(r_N)$ be as in Lemma~\ref{ineq of norms projection}. Then $\pi(r_N)f-1\in J_{\alpha,M,m}.$ It follows that
			\begin{align*}
				||r_Nf-1||_\alpha\geq ||\pi(r_N)f-1||_\alpha
				\asymp ||\tilde{\pi}(r_N)\tilde{f}-1||_{d_{\beta(\alpha)}}
				\geq \mathrm{dist}_{d_{\beta(\alpha)}}(1,\tilde{f}\cdot P^1_N),
			\end{align*}
			since $\tilde{\pi}(r_N)\in P^1_N.$ On the other hand, let
			$$\mathrm{dist}_{d_{\beta(\alpha)}}(1,\tilde{f}\cdot P^1_N)=||q_N\tilde{f}-1||_{d_{\beta(\alpha)}}, \quad q_N(z)=\sum_{l=0}^{N}a_lz^l.$$
			Then, the polynomial
			$$q'_N(z_1,...,z_n)=\sum_{l=0}^{N}\mu(m)^{-l/4}a_lz_1^{M_1l}\cdots z_m^{M_ml}$$
			satisfies $q_N'\in J_{\alpha,M,m}\cap P_N^n$ and $q_N'f-1\in J_{\alpha,M,m}.$ Thus,
			\begin{align*}
				||q_N\tilde{f}-1||_{d_{\beta(\alpha)}}=||\tilde{q}_N'\tilde{f}-1||_{d_{\beta(\alpha)}}
				\asymp ||q_N'f-1||_\alpha
				\geq \mathrm{dist}_{D_\alpha(\BB)}(1,f\cdot P_N^n)
			\end{align*}
			and the assertion follows.
		\end{proof}
	\end{proposition}
	
	Define the function $\phi_{\beta}:[0,\infty)\rightarrow [0,\infty)$ by
	$$\phi_\beta(t)=
	\begin{dcases}
		t^{1-\beta}, & \beta<1 \\
		\log^+(t), & \beta=1
	\end{dcases}
	,$$
	where $\log^+(t):=\max\{\log t,0\}.$
	We have the following.
	\begin{theorem}
		Let $\alpha\in \RR$ be such that $\beta(\alpha)\leq 1.$ Let $f\in J_{\alpha,M,m}$ be as above and suppose the corresponding $\tilde{f}$ has no zeros inside its domain, has at least one zero on the boundary, and admit an analytic continuation to a strictly bigger domain. Then $f$ is cyclic in $D_\alpha(\BB)$ whenever $\alpha\leq \frac{2n-m+1}{2}$ and
		$$ \mathrm{dist}^2_{D_\alpha(\BB)}(1,f\cdot P_N^n)\asymp \phi_{\beta(\alpha)}(N+1)^{-1}.$$
		\begin{proof} It is an immediate consequence of the identification between $D_\beta(\DD)$ and $d_\beta$ and previous lemmas and propositions.
		\end{proof}
	\end{theorem}
	
	If we focus on polynomials, then the following is true.
	\begin{theorem}\label{Thm characterization of polynom}
		Consider the polynomial $p(z)=1-m^{m/2} z_1\cdots z_m,$ where $1\leq m\leq n.$ Then $p$ is cyclic in $D_\alpha(\BB)$ whenever $\alpha\leq\frac{2n+1-m}{2}.$
	\end{theorem}
	
	Note that the Theorem\ref{Thm characterization of polynom} is not a characterization. We shall study the case $\alpha>\frac{2n+1-m}{2}.$
	
	\section{Cyclicity for model polynomials via radial dilation}\label{section radial}
	
	The \emph{radial dilation} of a function $f:\BB\rightarrow \CC$ is defined for $r\in(0,1)$ by $f_r(z)=f(rz).$  
	To prove Theorem~\ref{Thm characterization of polynom}, it is enough to prove the following lemma.
	
	\begin{lemma}\label{missing} Consider the polynomial $p(z)=1-m^{m/2} z_1\cdots z_m,$ where $1\leq m\leq n.$ Then $||p/p_r||_{\alpha}<\infty$ as $r\rightarrow 1^{-}$  whenever $\alpha\leq\frac{2n+1-m}{2}.$
	\end{lemma}
	
	We follow the arguments of \cite{Kosinski Vav}, \cite{Knese}. Indeed, if Lemma~\ref{missing} holds, then $\phi_r\cdot p\rightarrow 1$ weakly, where $\phi_r:=1/p_r.$ This is a consequence of a crucial property of Dirichlet-type spaces: if $\{f_n\}\subset D_\alpha(\BB),$ then $f_n\rightarrow 0$ weakly if and only if $f_n\rightarrow 0$ pointwise and $\sup_{n}\{||f_n||_\alpha\}<\infty.$  Since $\phi_r$ extends holomorphically past the closed unit ball, $\phi_r$ are multipliers, and hence, $\phi_r\cdot p\in[p].$ Finally, $1$ is weak limit of $\phi_r\cdot p$ and $[p]$ is closed and convex or, equivalently, weakly closed. It is clear that $1\in [p],$ and hence, $p$ is cyclic.
	
	Moreover, it is enough to prove that $||p/p_r||_\alpha<\infty,$ as $r\rightarrow 1^{-},$ for $\alpha_0=\frac{2n+1-m}{2}.$ Then the case $\alpha<\alpha_0$ follows since the inclusion $D_{\alpha_0}(\BB)\hookrightarrow D_\alpha(\BB)$ is a compact linear map and weak convergence in $D_{\alpha_0}(\BB)$ gives weak convergence in $D_\alpha(\BB).$
	
	\begin{proof}[Proof of Lemma~\ref{missing}]
		By Theorem~\ref{equivalent asym norm with cauchy formula} it is enough to show the following:
		$$I_p:=\int_{\BB}(1-||z||^2)\Big|\int_{\mathbb S_n}\frac{(1-\lambda \zeta_1\cdots \zeta_m)\bar{\zeta}_p}{(1-r^m\lambda \zeta_1\cdots \zeta_m)(1-\langle z,\zeta \rangle)^{\beta}}d\sigma(\zeta)\Big|^2du(z)$$
		and
		$$I_{p,q}:=\int_{\BB}\Big|\int_{\mathbb S_n}\frac{(\overline{z_p\zeta_q-z_q\zeta_p})(1-\lambda \zeta_1\cdots \zeta_m)}{(1-r^m\lambda \zeta_1\cdots \zeta_m)(1-\langle z,\zeta \rangle)^{\beta}}d\sigma(\zeta)\Big|^2du(z)$$
		are finite, as $r\rightarrow 1^{-},$ where $\beta=n+t+1,$ $t=\frac{2n+1-m}{4},$ and $\lambda=m^{m/2}.$ 
		
		Denote
		$$S_p:=\int_{\mathbb S_n}\frac{(1-\lambda \zeta_1\cdots \zeta_m)\bar{\zeta}_p}{(1-r^m\lambda \zeta_1\cdots \zeta_m)(1-\langle z,\zeta \rangle)^{\beta}}d\sigma(\zeta),$$
		where the last integral is equal to
		$$\frac{1}{2\pi}\int_{\mathbb S_n}\int_{0}^{2\pi}\frac{(1-\lambda e^{im\theta}\zeta_1\cdots \zeta_m)e^{-i\theta}\bar{\zeta}_p}{(1-r^m\lambda e^{im\theta}\zeta_1\cdots \zeta_m)(1-e^{-i\theta}\langle z,\zeta \rangle)^\beta}d\theta d\sigma(\zeta).$$
		Let $z,\zeta$ be fixed. Then 
		$$ \int_{0}^{2\pi}\frac{e^{-i\theta}}{(1-e^{-i\theta}\langle z,\zeta \rangle)^\beta}d\theta=0.$$
		Thus, replacing $p(e^{i\theta}\zeta)/p(re^{i\theta}\zeta)$ by $p(e^{i\theta}\zeta)/p(re^{i\theta}\zeta)-1$ we obtain
		$$S_p=\frac{\lambda(r^m-1)}{2\pi}\int_{\mathbb S_n}\int_{0}^{2\pi}\frac{ \bar{\zeta}_p\zeta_1\cdots\zeta_me^{i(m-1)\theta}}{(1-r^m\lambda e^{im\theta}\zeta_1\cdots \zeta_m)(1-e^{-i\theta}\langle z,\zeta \rangle)^\beta}d\theta d\sigma(\zeta).$$
		
		Next, expand the binomials
		\begin{align*}
			\int_{0}^{2\pi}&\frac{e^{i(m-1)\theta}}{(1-r^m\lambda e^{im\theta}\zeta_1\cdots \zeta_m)(1-e^{-i\theta}\langle z,\zeta \rangle)^\beta}d\theta\\
			&=\sum_{k=0}^{\infty}\sum_{l=0}^{\infty}\frac{\Gamma(k+\beta)}{\Gamma(\beta)k!}(r^m\lambda\zeta_1\cdots \zeta_m)^l\langle z,\zeta \rangle^k\int_{0}^{2\pi}e^{i(m(l+1)-k-1)\theta}d\theta\\
			&=2\pi\sum_{k=0}^{\infty}\frac{\Gamma(m(k+1)-1+\beta)}{\Gamma(\beta)(m(k+1)-1)!}(r^m\lambda\zeta_1\cdots \zeta_m)^k\langle z,\zeta \rangle^{m(k+1)-1}\\
			&=2\pi \sum_{k=0}^{\infty}\sum_{|j|=m(k+1)-1}\frac{\Gamma(m(k+1)-1+\beta)}{\Gamma(\beta)j!}(r^m\lambda\zeta_1\cdots \zeta_m)^kz^j\bar{\zeta}^j.
		\end{align*}
		Therefore,
		$$S_p=\lambda(r^m-1)\sum_{k=0}^{\infty}\sum_{|j|=m(k+1)-1}\frac{\Gamma(m(k+1)-1+\beta)}{\Gamma(\beta)j!}(r^m\lambda)^kz^jc(k),$$
		where $c(k)=\int_{\mathbb{S}_n}\zeta^{\alpha(k)}\bar{\zeta}^{b(k)}d\sigma(\zeta),$ 
		$\alpha(k)=(k+1,..,k+1\text{(m-comp.)},0,..,0)$ and $b(k)=(j_1,...,j_{p-1},j_p+1,j_{p+1},...,j_n).$ Whence, $1\leq p\leq m.$ Since the holomorphic monomials are orthogonal to each other in $L^2(\sigma)$ we get that
		$$|S_p|\asymp(1-r^m)\Big|z'_p\sum_{k=0}^{\infty}(k+1)^{\beta-n}(r^m\lambda z_1\cdots z_m)^k\Big|,$$
		where $z'_p=z_1\cdots z_{p-1}z_{p+1}\cdots z_m.$
		Hence we obtain 
		$$I_p\asymp(1-r^m)^2\sum_{k=0}^{\infty}(k+1)^{2(\beta-n)}(r^m\lambda)^{2k}\int_{\BB}(1-||z||^2)|z'_p|^2|z_1\cdots z_m|^{2k}du(z),$$
		where has been used again the orthogonality of the holomorphic monomials in $L^2(\sigma).$ 
		To handle the integral above we use polar coordinates
		\begin{align*}
			\int_{\BB}&(1-||z||^2)|z'_p|^2|z_1\cdots z_m|^{2k}du(z)\\
			&\asymp \int_{0}^{1}\int_{\mathbb S_n}\epsilon^{2n-1}(1-\epsilon^2)\epsilon^{2km+2m-2}|\zeta_p'|^2|\zeta_1\cdots \zeta_m|^{2k}d\sigma(\zeta)d\epsilon\\
			&\asymp \frac{[(k+1)!]^{m-1}k!}{(n+m(k+1)-2)!}\cdot\frac{1}{(k+1)^2}.
		\end{align*}
		If we recall that $\beta=n+t+1,$ $t=\frac{2n+1-m}{4}$ and $\lambda^{2k}=m^{mk},$ then applying the Stirling formula more than one time we see that
		$$ I_p\asymp(1-r^m)^2 \sum_{k=0}^{\infty}(k+1)r^{2mk}.$$
		This proves the assertion made about $I_p.$ 
		
		It remains to estimate the following term:
		$$I_{p,q}=\int_{\BB}\Big|\int_{\mathbb S_n}\frac{(\overline{z_p\zeta_q-z_q\zeta_p})(1-\lambda \zeta_1\cdots \zeta_m)}{(1-r^m\lambda \zeta_1\cdots \zeta_m)(1-\langle z,\zeta \rangle)^{\beta}}d\sigma(\zeta)\Big|^2du(z).$$
		We shall show that $I_{p,q}\asymp I_p.$ Denote again the inner integral by $S_{p,q}$ which is convenient to  expand it as $S_{p,q}=\bar{z}_pS_q-\bar{z}_qS_p.$ Recall that $z_p'=z_1\cdots z_{p-1}z_{p+1}\cdots z_m.$ Similar calculations to the one above lead to
		$$|S_{p,q}|\asymp(1-r^m)|\bar{z}_pz_q'-\bar{z}_qz_p'|\Big|\sum_{k=0}^{\infty}(k+1)^{\beta-n}(r^m\lambda z_1\cdots z_m)^k\Big|.$$
		Moreover, the orthogonality of the holomorphic monomials in $L^2(\sigma)$ gives the following estimation:
		$$I_{p,q}\asymp (1-r^m)^2\sum_{k=0}^{\infty}(k+1)^{2\beta-2n}(r^m\lambda)^{2k}\int_{\BB}|\bar{z}_pz_q'-\bar{z}_qz_p'|^2|z_1\cdots z_m|^{2k}du(z).$$
		It is easy to see that $|\bar{z}_pz_q'-\bar{z}_qz_p'|^2=|z_p|^2|z'_q|^2+|z_q|^2|z'_p|^2-2|z_1\cdots z_m|^2.$ Let us estimate the integral
		$$\int_{\BB}(|z_p|^2|z_q'|^2-|z_1 \cdots z_m|^2)|z_1\cdots z_m|^{2k}du(z).$$
		Passing through polar coordinates we get, for $p\neq q,$ that
		\begin{align*}
			\int_{\BB}|z_p|^2|z_q'|^2&|z_1\cdots z_m|^{2k}du(z)\\
			&\asymp 2n(n-1)!\frac{[(k+1)!]^{m-1}k!}{(mk+n+m-1)!}\frac{k+2}{2km+2n+2m},
		\end{align*}
		and
		\begin{align*}
			\int_{\BB}|z_1\cdots z_m|^{2(k+1)}&du(z)\\
			&=2n(n-1)!\frac{[(k+1)!]^{m-1}k!}{(mk+n+m-1)!}\frac{k+1}{2km+2n+2m}.
		\end{align*}
		Hence we obtain
		\begin{align*}
			\int_{\BB}(|z_p|^2|z_q'|^2-|z_1 \cdots z_m|^2)&|z_1\cdots z_m|^{2k}du(z)\\
			&\asymp \frac{[(k+1)!]^{m-1}k!}{(mk+n+m-2)!(k+1)^2}.
		\end{align*}
		Again, applying the Stirling formula to the one above estimates we obtain
		$$I_{p,q}\asymp(1-r^m)^2 \sum_{k=0}^{\infty}(k+1)r^{2mk}.$$
		This proves the assertion made about $I_{p,q}.$
	\end{proof}
	
	\section{Sufficient conditions for non-cyclicity via Cauchy transforms and $\alpha$-capacities}\label{section capacity}
	
	We consider the $\emph{Cauchy transform}$ of a complex Borel measure $\mu$ on the unit sphere by
	$$C_{[\mu]}(z)=\int_{\mathbb S_n}\frac{1}{(1-\langle z,\bar{\zeta}\rangle )^n}d\mu(\zeta), \quad z\in \BB.$$
	Note that this definition differs from the classical one.
	
	Let $f\in D_\alpha(\BB)$ and put a measure $\mu$ on $\mathcal{Z}(f^*)$: the zero set in the sphere of the radial limits of $f$. The results in \cite{Sola} about Cauchy transforms and non-cyclicity are valid in our settings. We deduce that $[f] \neq D_\alpha(\BB)$, and hence non-cyclicity, whenever $C_{[\mu]}\in D_{-\alpha}(\BB).$ Thus, it is important to compute the Dirichlet-type norm of the Cauchy transform.
	
	Let $\mu$ be a Borel measure on $\mathbb{S}_n$ and set $\mu^*(j)=\int_{\mathbb S_n}\zeta^jd\mu(\zeta),$ $\bar{\mu}^*(j)=\int_{\mathbb S_n}\bar{\zeta}^jd\mu(\zeta).$  We have the following.
	\begin{lemma}\label{Cauchy norm}
		Let $\mu$ be a Borel measure on $\mathbb{S}_n.$ Then
		$$||C_{[\mu]}||_{-\alpha}^2\asymp \sum_{k=0}^{\infty}\sum_{|j|=k}\frac{(k+1)^{n-1-\alpha}k!}{j!}|\bar{\mu}^*(j)|^2.$$
		\begin{proof}
			Our Cauchy integral of $\mu$ on $\BB$ has the following expansion 
			$$C_{[\mu]}(z)=\sum_{k=0}^{\infty}\sum_{|j|=k}\frac{\Gamma(k+n)}{\Gamma(n)j!}\bar{\mu}^*(j)z^j.$$
			Therefore, one can compute the norm of $C_{[\mu]}$ in the space $D_{-\alpha}(\BB).$  The assertion follows.
		\end{proof}
	\end{lemma}
	
	The following lemma is crucial in the sequel. It is probably known, but we were not able to locate it in the literature, and hence we include its proof.
	\begin{lemma}\label{estimate with j}
		Let $j_1,...,j_n,k$ be non-negative integers satisfying  $j_1+...+j_n=nk.$ Then
		$$j_1!\cdots j_n!\geq (k!)^n.$$
		\begin{proof}
			The $\Gamma$-function is logarithmically convex, and hence, we may apply the Jensen inequality to it:
			$$\log \Gamma\Big(\frac{x_1}{n}+...+\frac{x_n}{n}\Big)\leq \frac{\log \Gamma(x_1)}{n}+...+\frac{\log \Gamma(x_n)}{n}.$$
			Set $x_i:=j_i+1,$ $i=1,...,n.$ Since $j_1+...+j_n=nk,$ the assertion follows.
		\end{proof}
	\end{lemma}
	
	We may identify non-cyclicity for model polynomials via Cauchy transforms.
	\begin{lemma}
		Consider the polynomial $p(z)=1-m^{m/2} z_1\cdots z_m,$ where $1\leq m\leq n.$ Then $p$ is not cyclic in $D_\alpha(\BB)$ whenever $\alpha>\frac{2n+1-m}{2}.$
		\begin{proof}
			Recall that the model polynomials vanish in the closed unit ball along analytic sets of the form:
			$$\mathcal{Z}(p)\cap \mathbb{S}_n=\{1/\sqrt{m}(e^{i\theta_1},..,e^{i\theta_{m-1}}, e^{-i(\theta_1+...+\theta_{m-1})},0,..,0):\theta_i\in \mathbb{R}\}.$$
			It is easy to see that for a proper measure $\mu$, $\mu^*(j)$ is non-zero when $mj_m=k$ and $\mu^*(j)\asymp m^{-k/2}.$ By Stirling's formula and Lemma~\ref{estimate with j} we get that
			\begin{align*}
				||C_{[\mu]}||_{-\alpha}^2&\leq  C \sum_{k=0}^{\infty}\frac{(mk+1)^{n-1-\alpha}(mk)!}{(k!)^mm^{mk}}\asymp \sum_{k=0}^{\infty}(k+1)^{1/2(2n-m-1)-\alpha}.
			\end{align*}
			Thus, $p$ is not cyclic in $D_\alpha(\BB)$ for $\alpha>\frac{2n+1-m}{2}.$
		\end{proof}
	\end{lemma}
	
	We consider Riesz $\alpha$-capacity for a fixed parameter $\alpha\in(0,n)$ with respect to the \emph{anisotropic distance} in $\mathbb S_n$ given by 
	$$d(\zeta,\eta)=|1-\langle \zeta,\eta \rangle|^{1/2}$$
	and the non-negative \emph{kernel} $K_\alpha:(0,\infty)\rightarrow [0,\infty)$ given by
	
	$$K_\alpha(t)=
	\begin{dcases}
		t^{\alpha-n}, & \alpha\in(0,n) \\
		\log(e/t), & \alpha=n
	\end{dcases}
	.$$
	Note that we may extend the definition of $K$ to $0$ by defining $K(0):=\lim_{t\rightarrow 0^+}K(t).$
	
	Let $\mu$ be any Borel probability measure supported on some Borel set $E\subset \mathbb{S}_n.$ Then the  $\emph{Riesz}$ $\alpha$-$\emph{energy}$ of $\mu$ is given by
	$$I_\alpha[\mu]=\iint_{\mathbb S_n}K_\alpha(|1-\langle \zeta,\eta \rangle|)d\mu(\zeta)d\mu(\eta)$$
	and the  $\emph{Riesz}$ $\alpha$-$\emph{capacity}$ of $E$ by
	$$\mathrm{cap}_\alpha(E)=\inf\{I_\alpha[\mu]:\mu\in \mathcal{P}(E)\}^{-1},$$
	where $\mathcal{P}(E)$ is the set of all Borel probability measures supported on $E.$ Note that if $\text{cap}_\alpha(E)>0,$ then there exist at least one probability measure supported on $E$ having finite Riesz $\alpha$-energy. Moreover, any $f\in D_\alpha(\BB)$ has finite radial limits $f^*$ on $\mathbb{S}_n,$ except possibly, on a set $E$ having $\text{cap}_\alpha(E)=0.$ Theory regarding to the above standard construction in potential theory can be found in \cite{Ahern}, \cite{Cohn}, \cite{Fallah-Tomas}, \cite{Pestana}. 
	
	The relation between non-cyclicity of a function and the Riesz $\alpha$-capacity of the zeros of its radial limits follows.
	\begin{theorem}\label{Capacity theorem non cyclic}
		Fix $\alpha\in (0,n]$ and let $f\in D_\alpha(\BB).$ If $\mathrm{cap}_\alpha(\mathcal{Z}(f^*))>0,$ then $f$ is not cyclic in $D_\alpha(\BB).$
		\begin{proof}
			Let $\mu$ be a probability measure supported in $\mathcal{Z}(f^*),$ with finite Riesz $n$-energy. If $r\in (0,1),$ then
			\begin{align*} 
				\log\frac{e}{|1-r\langle \zeta,\eta\rangle|}&=1+\textrm{Re}\Big(\log\frac{1}{1-r\langle \zeta,\eta\rangle}\Big)\\
				&=1+\textrm{Re} \sum_{k=1}^{\infty}\sum_{|j|=k}\frac{r^kk!}{kj!}\zeta^j\overline{\eta}^j.
			\end{align*}
			Note that $\mu$ is a probability measure and hence
			\begin{align*}
				\iint_{\mathbb S_n}\log\frac{e}{|1-r\langle \zeta,\eta\rangle|}d\mu(\zeta)d\mu(\eta)
				=1+ \sum_{k=1}^{\infty}\sum_{|j|=k}\frac{r^kk!}{kj!}|\mu^*(j)|^2.
			\end{align*}
			Since $|1-w|/|1-rw|\leq 2$ for $r\in (0,1)$ and $w\in \overline{\DD},$ the dominated convergence theorem and Lemma~\ref{Cauchy norm} give
			$$||C_{[\mu]}||^2_{-n}\asymp\sum_{k=1}^{\infty}\sum_{|j|=k}\frac{k!}{kj!}|\mu^*(j)|^2<\infty.$$
			The assertion follows.
			
			We continue setting a probability measure $\mu,$ supported in $\mathcal{Z}(f^*),$ with finite Riesz $\alpha$-energy, where $\alpha\in (0,n).$ If $r\in (0,1),$ then
			\begin{align*}
				\frac{1}{(1-r\langle \zeta,\eta\rangle)^{n-\alpha}}&=\sum_{k=0}^{\infty}\sum_{|j|=k}\frac{\Gamma(k+n-\alpha)k!r^k}{k!\Gamma(n-\alpha)j!}\zeta^j\overline{\eta}^j.
			\end{align*}
			Similar arguments to the one above show that
			\begin{align*}
				I_\alpha[\mu]&\geq\Big|\iint_{\mathbb S_n}\textrm{Re}\Big(\frac{1}{(1-r\langle \zeta,\eta\rangle)^{n-\alpha}}\Big)d\mu(\zeta)d\mu(\eta)\Big|\\
				&=\Big|\sum_{k=0}^{\infty}\sum_{|j|=k}\frac{\Gamma(k+n-\alpha)k!r^k}{k!\Gamma(n-\alpha)j!}\iint_{\mathbb S_n}\zeta^j\overline{\eta}^jd\mu(\zeta)d\mu(\eta)\Big|\\
				&\asymp \sum_{k=0}^{\infty}\sum_{|j|=k}\frac{(k+1)^{n-1-\alpha}k!}{j!}r^k|\mu^*(j)|^2.
			\end{align*}
			Again, letting $r\rightarrow 1^{-}$ by Lemma~\ref{Cauchy norm} we obtain that $C_{[\mu]}\in D_{-\alpha}(\BB).$ The assertion follows.
		\end{proof}
	\end{theorem}
	
	\begin{remark}
		According to \cite{Kosinski Vav} one can expect that the cyclicity problem of polynomials in the unit ball of $\CC^n$ depends on the real dimension of their zero set restricted on the unit sphere: $\text{dim}_{\RR}(\mathcal{Z}(p)\cap \mathbb S_n).$
		
		Let us point out the nature of the boundary zeros of a polynomial non-vanishing in the ball. See \cite{Kosinski Vav} for the two dimensional case where had been used the Curve Selection Lemma of \cite{Selection lemma}.
		
		Let $p\in \CC[z_1,...,z_n]$ be a polynomial non-vanishing in the ball. Looking at $\mathcal{Z}(p)\cap \mathbb S_n$ as at a semi-algebraic set, we conclude that it is the disjoint union of a finite number of Nash manifolds $M_i,$ each Nash diffeomorphic to an open hypercube $(0,1)^{\textrm{dim}(M_i)}.$ Note that the Nash diffeomorphisms over the closed field of the real numbers satisfy some additional properties (see \cite{real algebraic geom}, Proposition 2.9.10). 
		
		One can expect then that the characterization of cyclicity and the nature of the boundary zeros of the model polynomials, as well as, the unitary invariance of the Dirichlet norm and the sufficient capacity condition, will be crucial in the characterization of cyclic polynomials in arbitrary dimension.
	\end{remark}
	
	{\bf Acknowledgments.}  I would like to thank \L{}. Kosi\'{n}ski for the helpful conversations during the preparation of the present work. I would like to thank also the anonymous referee for numerous remarks that substantially improved the shape of the paper.

\end{document}